\documentclass[12pt]{amsart}

\usepackage{ucs}

\usepackage{amssymb}
\usepackage{amsthm}
\usepackage{amsmath}
\usepackage{latexsym}
\usepackage[cp1251]{inputenc}
\usepackage[mathcal]{eucal}
\usepackage{graphicx}
\usepackage{wrapfig}
\usepackage{caption}
\usepackage{subcaption}
\usepackage{indentfirst}
\usepackage[left=2.6cm,right=2.6cm,top=3cm,bottom=3cm,bindingoffset=0cm]{geometry}
\usepackage{enumerate}
\usepackage{color}

\DeclareMathOperator{\aut}{Aut}

\DeclareMathOperator{\diag}{Diag}

\DeclareMathOperator{\cay}{Cay}
\DeclareMathOperator{\cyc}{Cyc}

\DeclareMathOperator{\iso}{Iso}

\DeclareMathOperator{\orb}{Orb}

\DeclareMathOperator{\rk}{rk}

\DeclareMathOperator{\Span}{Span}

\DeclareMathOperator{\sym}{Sym}
\DeclareMathOperator{\rad}{rad}

\title[$CI$-property for decomposable Schur rings]{$CI$-property for decomposable Schur rings over an abelian group}

\author{Istv\'an Kov\'acs}
\address{University of Primorska, Koper, Slovenia}
\email{istvan.kovacs@upr.si}
\author{Grigory Ryabov}
\address{Sobolev Institute of Mathematics, Novosibirsk, Russia}
\address{Novosibirsk State University, Novosibirsk, Russia}
\email{gric2ryabov@gmail.com}
\thanks{I.~Kov\'acs was supported in part by 
the Slovenian Research Agency (research program P1-0285 and research projects N1-0032, N1-0038, N1-0062, J1-7051 and J1-9108). G.~Ryabov was supported in part by RSF (project No. 14-21-00065), and is grateful to the University 
of Primorska for hospitality.}

\date{}

\newtheorem{prop}{Proposition}[section]

\newtheorem*{theo1}{Theorem 1}

\newtheorem{lemm}[prop]{Lemma}

\theoremstyle{definition}

\begin{document}

\begin{abstract}
A Schur ring over a finite group is said to be \emph{decomposable} if it is the generalized wreath product of  Schur rings over smaller groups. In this paper we establish a sufficient condition for a decomposable Schur ring over the direct product of elementary abelian groups to be a $CI$-Schur ring. By using this condition we reprove in a short way known results on the $CI$-property for decomposable Schur rings  over  an elementary abelian group of rank at most~$5$.
\end{abstract}

\maketitle

\section{Introduction}

Let $G$ be a finite group. An  \emph{$S$-ring} (a \emph{Shur ring}) over  $G$ is defined to be a subring of the integer group ring $\mathbb{Z}G$ which is a free $\mathbb{Z}$-module spanned by a partition of $G$ closed under taking inverse and containing the identity element of $G$ as a class (the exact definition is given in Section~2). The concept of an $S$-ring goes back to Schur and Wielandt who studied a permutation group containing a regular subgroup~\cite{Schur,Wi}. An $S$-ring over~$G$ is
called \emph{schurian} if it is associated in a natural way with a subgroup of $\sym(G)$ that contains all right translations.

Let $\mathcal{A}$ and $\mathcal{A}^{'}$ be $S$-rings over groups $G$  and  $G^{'},$ respectively. A \emph{(combinatorial) isomorphism } from $\mathcal{A}$ to $\mathcal{A}^{'}$  is defined to be a bijection  from $G$ to $G^{'}$  that is an isomorphism of the corresponding Cayley schemes $\mathcal{C}(\mathcal{A})$ and $\mathcal{C}(\mathcal{A}^{'})$. Set 
$$\iso(\mathcal{A})=\{f\in \sym(G):~\text{f is an isomorphism from}~\mathcal{A}~\text{onto}~\text{$S$-ring over}~G\}.$$
An isomorphism from $\mathcal{A}$ onto itself is called an \emph{automorphism} of $\mathcal{A}$ if it preserves every basic relation of $\mathcal{C}(\mathcal{A})$. All automorphisms of $\mathcal{A}$ form a group called the \emph{automorphism group} of $\mathcal{A}$ and denoted by $\aut(\mathcal{A})$. An $S$-ring $\mathcal{A}$ is called a \emph{$CI$-$S$-ring} if $\iso(\mathcal{A})=\aut(\mathcal{A})\aut(G)$. This definition was suggested by Hirasaka and Muzychuk in~\cite{HM}. An importance of $CI$-$S$-rings arises from the following result of this paper: if every schurian $S$-ring over a group $G$ is a $CI$-$S$-ring then $G$ is a $DCI$-group.

Recall that a set $S\subseteq G$ is  a \emph{$CI$-subset} if for every $T\subseteq G$ the isomorphism of Cayley graphs $\cay(G,S)$ and $\cay(G,T)$ implies that $T=S^{\varphi}$ for some $\varphi \in \aut(G)$. A group $G$ is said to be a \emph{$DCI$-group} if each of its subsets is a $CI$-subset and $G$ is said to be a \emph{$CI$-group} if each of its inverse-closed subsets is a $CI$-subset. One can check that a subgroup of a $DCI$-group ($CI$-group) is also a $DCI$-group ($CI$-group). In~\cite{Adam}, \'Ad\'am conjectured that every cyclic group is a $CI$-group. However, this conjecture turned  out
to be false. In~\cite{BF}, Babai and Frankl asked the following question: which are the $CI$-groups? %The 
 Most of the results on $DCI$- and $CI$-groups can be found in the survey paper~\cite{Li}.
%It should be mentioned that the classes of $DCI$- and $CI$-groups are closed under taking subgroups.

Denote the cyclic group of order~$n$ by $C_n$. From~\cite[Theorem~8.8]{Li}  it follows that every Sylow subgroup of an abelian $DCI$-group is elementary abelian or isomorphic to $C_4$. %One of the main steps towards the classification of all $CI$-groups is the classification of all elementary abelian $CI$-groups (\cite[Question 8.3]{Li}). 
Let $p$ and $q$ be distinct primes. The following abelian groups are $DCI$-groups: $C_p$~\cite{ET}; $C_p^2$, $C_p^3$~\cite{AlN}; $C_2^4$, $C_2^5$~\cite{CLi}; $C_p^4$, where $p$ is odd~\cite{HM}; $C_p^5$, where $p$ is odd \cite{FK}; $C_k$, $C_{2k}$, $C_{4k}$, where $k$ a square-free odd number~\cite{M}; $C_p^2\times C_q$~\cite{MK}. On the other hand, the following groups are not $CI$-groups: $C_2^n$ for $n\geq 6$~\cite{Now}; $C_3^n$ for $n\geq 8$~\cite{Sp1}; $C_p^n$ for $n\geq 2p+3$~\cite{Som}. 

The proof of the fact that the group $G=C_p^n$, where $p$ is an odd prime and $n\in\{4,5\}$, is a $DCI$-group is based on the method of $S$-rings. In fact, in this proof it was checked that every schurian $S$-ring over $G$ is a $CI$-$S$-ring. Due to the result of Hirasaka and Muzychuk, this is sufficient for  the proof that $G$ is a $DCI$-group. One of the main difficulties here was to check that every decomposable schurian $S$-ring over $G$ is a $CI$-$S$-ring. Recall that an $S$-ring $\mathcal{A}$ is said to be \emph{decomposable} if it is the $U/L$-wreath product of $S$-rings $\mathcal{A}_U$ and $\mathcal{A}_{G/L}$ for some $\mathcal{A}$-section $U/L$ of $G$ with 
%the nontrivial $L$ 
 $1 < |L|$ and $U<G$ (see Subsection~2.1 for exact definitions). The main goal of this paper is to find a sufficient condition for a decomposable $S$-ring over the direct product of elementary abelian groups to be a $CI$-$S$-ring, and give short proofs of the known results on the $CI$-property for decomposable $S$-rings over an elementary abelian group of rank at most~$5$.

Again, let $G$ and $G^{'}$ be finite groups. For a set $\Delta\subseteq \sym(G)$ and a section $S=U/L$ of~$G$ we set $\Delta^S=\{f^S:~f\in \Delta,~S^f=S\}$, where $S^f=S$ means that $f$  maps 
$U$ to itself and it permutes the $L$-cosets in $U,$ and $f^S$ denotes the bijection of $S$ induced by $f$. 
Let $\mathcal{A}$ be an $S$-ring over~$G$. Put $\aut_G(\mathcal{A})=\aut(\mathcal{A})\cap \aut(G)$. For every $\mathcal{A}$-section $S$ of $G$ one can from the \emph{quotient} $S$-ring $\mathcal{A}_S$ over $S$ (see Subsection~2.1). Each $f\in \aut_G(\mathcal{A})$ induces %the 
 a combinatorial automorphism of $\mathcal{A}_S,$ 
%and
 which is also a %the
group automorphism of $S$. So $\aut_G(\mathcal{A})^S\leq\aut_S(\mathcal{A}_S)$. The main result of the paper is given in the theorem below.

\begin{theo1}\label{main}
Let $G$ be %the
 a direct product of elementary abelian groups, $\mathcal{A}$  be an $S$-ring over $G$, and 
$S=U/L$  be an 
$\mathcal{A}$-section of $G$. Suppose that $\mathcal{A}$ is the nontrivial $S$-wreath product and the $S$-rings $\mathcal{A}_U$ and $\mathcal{A}_{G/L}$ are $CI$-$S$-rings. Then $\mathcal{A}$ is a $CI$-$S$-ring whenever 
$$\aut_{S}(\mathcal{A}_{S})=\aut_U(\mathcal{A}_U)^{S}\aut_{G/L}(\mathcal{A}_{G/L})^{S}.~\eqno(1)$$ 
In particular,  $\mathcal{A}$ is a $CI$-$S$-ring if $\aut_{S}(\mathcal{A}_{S})=\aut_U(\mathcal{A}_U)^{S}$ or $\aut_{S}(\mathcal{A}_{S})=\aut_{G/L}(\mathcal{A}_{G/L})^{S}$.
\end{theo1}

We do not know whether the Condition~(1) is a necessary condition for an $S$-ring to be a $CI$-$S$-ring. If $U=L$ then, obviously, $\aut_{S}(\mathcal{A}_{S})$ is trivial and Condition~(1) holds. So Theorem~\ref{main} is a criterion for the groups $C_p^2$ and $C_p\times C_q$, where $p$ and $q$ are distinct primes, because in these cases $U$ must coincide with $L$. The computer calculations made by 
%using
~\cite{GAP} %implies 
 shows that Theorem~\ref{main} is a criterion for the groups $C_2^3$, $C_3^3$, $C_2^2\times C_3$, and $C_2\times C_3^2$.

The %text of the 
paper is organized in the following way. Section~2 contains a background of $S$-rings, especially, isomorphisms of $S$-rings, $p$- and $CI$-$S$-rings. In Section~3 we prove Theorem~\ref{main}. In Section~4 we give some corollaries of Theorem~\ref{main}. Finally, in Section~5 we use Theorem~\ref{main} to check the $CI$-property for decomposable $S$-rings over an elementary abelian group of rank at most~$5$.

The authors would like to thank the anonymous referee  for his constructive comments which helped us to improve the text significantly. 

\section{Preliminaries}

In this section we present some facts and definitions concerned with $S$-rings, %. The 
most of them can be  found in~\cite{FK,MP}. Throughout this section $G$ is a finite group and $e$ is the identity of $G$. The set of all orbits of a group $K$ acting on a set $\Omega$ is denoted by $\orb(K,\Omega)$.

\subsection{$S$-rings: basic facts and definitions}

Let  $\mathbb{Z}G$ be the integer group ring of~$G$. If $X\subseteq G$ then  denote the sum $\sum_{x\in X} {x}$  by $\underline{X}$. The set $\{x^{-1}:x\in X\}$ is denoted by $X^{-1}$. A subring  $\mathcal{A}\subseteq \mathbb{Z} G$ is called an \emph{$S$-ring} over $G$ if there exists a partition $\mathcal{S}(\mathcal{A})$ of~$G$ such that:

$(1)$ $\{e\}\in\mathcal{S}(\mathcal{A})$,

$(2)$  if $X\in\mathcal{S}(\mathcal{A})$ then $X^{-1}\in\mathcal{S}(\mathcal{A})$,

$(3)$ $\mathcal{A}=\Span_{\mathbb{Z}}\{\underline{X}:\ X\in\mathcal{S}(\mathcal{A})\}$.\\
The elements of $\mathcal{S}(\mathcal{A})$ are called the \emph{basic sets} of  $\mathcal{A}$ and the number $|\mathcal{S}(\mathcal{A})|$ is called the \emph{rank} of  $\mathcal{A}$. If $X,Y,Z\in\mathcal{S}(\mathcal{A})$ then   the number of distinct representations of $z\in Z$ in the form $z=xy$ with $x\in X$ and $y\in Y$ is denoted by $c^Z_{X,Y}$. Note that if $X$ and $Y$ are basic sets of $\mathcal{A}$ then $\underline{X}~\underline{Y}=\sum_{Z\in \mathcal{S}(\mathcal{A})}c^Z_{X,Y}\underline{Z}$. Therefore, the numbers  $c^Z_{X,Y}$ are the structure constants of $\mathcal{A}$ with respect to the basis $\{\underline{X}:\ X\in\mathcal{S}\}$. It is easy to check that given basic  sets $X$ and $Y$ the set $XY$ is also basic whenever  $|X|=1$ or $|Y|=1$.

Let $K$ be a subgroup of $\sym(G)$ containing  the group of right translations $G_{right}=\{x\mapsto xg,~x\in G:g\in G\}$. Let $K_e$ stand for the stabilizer of $e$ in $K$. Schur proved in  \cite{Schur} that the $\mathbb{Z}$-submodule
$$V(K,G)=\Span_{\mathbb{Z}}\{\underline{X}:~X\in \orb(K_e,~G)\},$$
is an $S$-ring over $G$. An $S$-ring $\mathcal{A}$ over  $G$ is called \emph{schurian} if $\mathcal{A}=V(K,G)$ for some $K$ such that $G_{right}\leq K \leq \sym(G)$. It should be mentioned that not every $S$-ring is schurian (see~\cite[Theorem~25.7]{Wi}).

Let $\mathcal{A}$ be an $S$-ring over $G$. A set $X \subseteq G$ is called an \emph{$\mathcal{A}$-set} if $\underline{X}\in \mathcal{A}$. A subgroup $H \leq G$ is called an \emph{$\mathcal{A}$-subgroup} if $H$ is an $\mathcal{A}$-set. Let $L \unlhd U\leq G$. %A 
 The section $U/L$ of $G$ is called an \emph{$\mathcal{A}$-section} if $U$ and $L$ are $\mathcal{A}$-subgroups. If $S=U/L$ is an $\mathcal{A}$-section of $G$ then the module
$$\mathcal{A}_S=Span_{\mathbb{Z}}\left\{\underline{X}^{\pi}:~X\in\mathcal{S}(\mathcal{A}),~X\subseteq U\right\},$$
where $\pi:U\rightarrow U/L$ is the canonical epimorphism, is an $S$-ring over $S$ called the \emph{quotient} $S$-ring. If $\mathcal{A}=V(K,G)$ for some $K\leq \sym(G)$ containig $G_{right}$ and $S$ is an $\mathcal{A}$-section of $G$ then $\mathcal{A}_S=V(K^S,G)$.

If $X \subseteq G$ then the set $\{g\in G:~Xg=gX=X\}$ is called the \emph{radical} of $X$ and denoted by $\rad(X)$. Clearly, $\rad(X)$ is a subgroup of $G$. If $X$ is an $\mathcal{A}$-set  then the groups $\langle X \rangle$ and $\rad(X)$ are $\mathcal{A}$-subgroups of~$G$. By the \emph{thin radical} of the $S$-ring $\mathcal{A}$ we mean the set 
$$O_{\theta}(\mathcal{A})=\{g\in G:~\{g\}\in \mathcal{S}(\mathcal{A})\}.$$
It is easy to see that $O_{\theta}(\mathcal{A})$ is an $\mathcal{A}$-subgroup.

Given $X \subseteq G$ and $m\in \mathbb{Z}$ put $X^{(m)}=\{x^m: x \in X\}$. The following statement is known  as Schur's  theorem on multipliers (see~\cite{Schur}).

\begin{lemm} \label{burn}
Let $\mathcal{A}$ be an $S$-ring over an abelian group  $G$. Then  $X^{(m)}\in \mathcal{S}(\mathcal{A})$  for every  $X\in \mathcal{S}(\mathcal{A})$ and every  $m\in \mathbb{Z}$ coprime to $|G|$.
\end{lemm}

The $S$-ring $\mathcal{A}$ over  $G$ is said to be \emph{cyclotomic} if there exists $M\le\aut(G)$ such that $\mathcal{S}(\mathcal{A})=\orb(M,G)$. In this case $\mathcal{A}$ is denoted by $\cyc(M,G)$. Obviously, $\mathcal{A}=V(G_{right}M,G)$. So every cyclotomic $S$-ring is schurian. If $\mathcal{A}=\cyc(M,G)$ for some $M\leq \aut(G)$ and $S$ is an $\mathcal{A}$-section of $G$ then $\mathcal{A}_S=\cyc(M^S,G)$.

Let $S=U/L$ be an $\mathcal{A}$-section of~$G$. The $S$-ring~$\mathcal{A}$ is called the \emph{$S$-wreath product} if $L\trianglelefteq G$ and $L\leq\rad(X)$ for all basic sets $X$ outside~$U$. In this case we write $\mathcal{A}=\mathcal{A}_U\wr_{S}\mathcal{A}_{G/L}$ and omit $S$ when $U=L$. An $S$-ring $\mathcal{A}$ is said to be the \emph{generalized wreath product} if $\mathcal{A}$ is the $S$-wreath product for some $\mathcal{A}$-section $S$ of $G$. The construction of the generalized wreath product for $S$-rings was introduced  in~\cite{EP7}. The $S$-wreath product is called \emph{nontrivial} or \emph{proper}  if $ \{e\} \neq L$ and $U\neq G$.  Note that  $\mathcal{A}$ can be reconstructed uniquely from the $S$-rings $\mathcal{A}_U$ and $\mathcal{A}_{G/L}$. We say that the $S$-ring $\mathcal{A}$ is \emph{decomposable} if $\mathcal{A}$ is the nontrivial $S$-wreath product for some $\mathcal{A}$-section $S$ of $G$ and  $\mathcal{A}$ is \emph{indecomposable} otherwise. Throughout the paper we consider only nontrivial generalized wreath products and further we will avoid the word ``nontrivial'' for short. 

If  $\mathcal{A}_1$ and $\mathcal{A}_2$ are $S$-rings over groups $G_1$ and $G_2,$ respectively, then the set
	
	$$\mathcal{S}=\mathcal{S}(\mathcal{A}_1)\times \mathcal{S}(\mathcal{A}_2)=\{X_1\times X_2:~X_1\in \mathcal{S}(\mathcal{A}_1),~X_2\in \mathcal{S}(\mathcal{A}_2)\} $$
forms a partition of  $G=G_1\times G_2$ that defines an  $S$-ring over $G$. This $S$-ring is called the  \emph{tensor product}  of $\mathcal{A}_1$ and $\mathcal{A}_2$ and denoted by $\mathcal{A}_1 \otimes \mathcal{A}_2$.

\begin{lemm}{\em\cite[Lemma 2.8]{FK}}\label{tenspr}
Let $\mathcal{A}$ be an $S$-ring over an abelian group $G=G_1\times G_2$. Assume that $G_1$ and $G_2$ are $\mathcal{A}$-groups. Then $\mathcal{A}=\mathcal{A}_{G_1}\otimes \mathcal{A}_{G_2}$ whenever $\mathcal{A}_{G_1}$ or $\mathcal{A}_{G_2}$ is the group ring.
\end{lemm}

\subsection{Isomorphisms of $S$-rings}
Throughout this subsection we follow~\cite{EP4,MP}.

Let $\mathcal{R}$ be a partition of $G\times G$. A pair $\mathcal{C}=\left(G,\mathcal{R}\right)$ is called a \emph{Cayley scheme} over $G$ if the following properties hold:

$(1)$ $\diag(G\times G)=\{(g,g):g\in G\}\in\mathcal{R}$;

$(2)$ if  $R\in\mathcal{R}$ then $R^*=\{(h,g): (g,h)\in R\}\in\mathcal{R}$;

$(3)$ if $R,~S,~T\in\mathcal{R}$ then the number $c^T_{R,S}=|\{h\in G:(g,h)\in R,~(h,f)\in S\}|$ does not depend on the choice of  $(g,f)\in T$;

$(4)$ $\{(hg,fg):(h,f)\in R\}=R$ for every $R\in\mathcal{R}$ and every $g\in G$.

There is a one-to-one correspondence between  $S$-rings and Cayley schemes over $G$. If $\mathcal{A}$ is an $S$-ring over $G$ then the  pair $\mathcal{C}(\mathcal{A})=\left(G,\mathcal{R}(\mathcal{A})\right)$, where $\mathcal{R}(\mathcal{A})=\{R(X):X\in \mathcal{S}(\mathcal{A})\}$ with $R(X)=\{(g,xg): g\in G, x\in X\}$, is a Cayley scheme over~$G$.

Let  $\mathcal{A}$  and $\mathcal{A}^{'}$ be $S$-rings over groups $G$  and $G^{'},$ respectively, and $\mathcal{C}=(G,\mathcal{R})$ and $\mathcal{C}^{'}=(G^{'},\mathcal{R}^{'})$  Cayley schemes over $G$ and $G^{'},$ respectively. A \emph{(combinatorial) isomorphism} from  $\mathcal{C}$ to  $\mathcal{C}^{'}$  is defined to be a bijection $f:G\rightarrow G^{'}$ such that $\mathcal{R}^{'}=\mathcal{R}^f$, where $\mathcal{R}^f=\{R^f:~R\in\mathcal{R}\}$ with $R^f=\{(g^f,~h^f):~(g,~h)\in R\}$. A \emph{(combinatorial) isomorphism} from  $\mathcal{A}$  to  $\mathcal{A}^{'}$  is defined to be a bijection $f:G\rightarrow G^{'}$  which is an isomorphism of the corresponding Cayley schemes  $\mathcal{C}(\mathcal{A})$ and $\mathcal{C}(\mathcal{A}^{'})$. The group $\iso(\mathcal{A},\mathcal{A})$  of all isomorphisms from $\mathcal{A}$ onto itself has a normal subgroup
$$\aut(\mathcal{A})=\{f\in \iso(\mathcal{A}): R(X)^f=R(X)~\text{for every}~X\in \mathcal{S}(\mathcal{A})\}.$$
This subgroup is called the \emph{automorphism group} of $\mathcal{A}$ and denoted by $\aut(\mathcal{A})$; the elements of $\aut(\mathcal{A})$ are called \emph{automorphisms} of $\mathcal{A}$. It is easy to see that $G_{right}\leq \aut(\mathcal{A})$.

An \emph{algebraic isomorphism} from $\mathcal{A}$  to $\mathcal{A}^{'}$ is defined to be a bijection $\varphi:\mathcal{S}(\mathcal{A})\rightarrow\mathcal{S}(\mathcal{A}^{'})$ such that $c_{X,Y}^Z=c_{X^{\varphi},Y^{\varphi}}^{Z^{\varphi}}$ for all $X,Y,Z\in \mathcal{S}(\mathcal{A})$. The mapping $\underline{X}\rightarrow \underline{X}^{\varphi}$ is extended by linearity to %the 
 a ring isomorphism from $\mathcal{A}$  to $\mathcal{A}^{'}$. It can be checked that every combinatorial isomorphism of $S$-rings preserves structure constants and hence induces the algebraic isomorphism. However, not every algebraic isomorphism is induced by a combinatorial one. Note that $f\in\iso(\mathcal{A},\mathcal{A})$ induces the trivial algebraic isomorphism if and only if $f\in\aut(\mathcal{A})$. 

Every algebraic isomorphism $\varphi:\mathcal{A}\rightarrow \mathcal{A}^{'}$  is extended to %the 
 a bijection between  $\mathcal{A}$- and $\mathcal{A}^{'}$-sets and hence between  $\mathcal{A}$- and $\mathcal{A}^{'}$-sections. Denote  the images of an $\mathcal{A}$-set $X$ and an $\mathcal{A}$-section $S$ under  $\varphi$  by $X^{\varphi}$ and $S^{\varphi}$ respectively. If $S$ is an $\mathcal{A}$-section then  $\varphi$ induces %the 
an algebraic isomorphism $\varphi^S:\mathcal{A}_S\rightarrow \mathcal{A}^{'}_{S^{'}}$, where $S^{'}=S^{\varphi}$.

\begin{lemm} {\em\cite[Theorem 3.3, (1)]{EP4}}\label{simgwr}
Let $\mathcal{A}$  and $\mathcal{A}^{'}$ be $S$-rings over abelian groups $G$ and $G^{'},$ respectively, and $U/L$  be an $\mathcal{A}$-section of $G$. Suppose that $\mathcal{A}$ is the $U/L$-wreath product, $\varphi$ is an algebraic isomorphism from $\mathcal{A}$ to $\mathcal{A}^{'}$, $U^{'}=U^{\varphi}$, and $L^{'}=L^{\varphi}$. Then $\mathcal{A}^{'}$ is the  $U^{'}/L^{'}$-wreath product.
\end{lemm}

A \emph{Cayley isomorphism} from  $\mathcal{A}$  to $\mathcal{A}^{'}$   is defined to be a group isomorphism $f:G\rightarrow G^{'}$ such that $\mathcal{S}(\mathcal{A})^f=\mathcal{S}(\mathcal{A}^{'})$. If there exists a Cayley isomorphism from  $\mathcal{A}$  to $\mathcal{A}^{'}$ we say that $\mathcal{A}$ and $\mathcal{A}^{'}$ are \emph{Cayley isomorphic} and write $\mathcal{A}\cong_{\cay}\mathcal{A}^{'}$. Every Cayley isomorphism is a (combinatorial) isomorphism however the converse statement is not true. %Set $\aut_G(\mathcal{A})=\aut(\mathcal{A}) \cap \aut(G)$. 

Let $f$ be a combinatorial isomorphism from $\mathcal{A}$ to $\mathcal{A}^{'}$.  Denote by $\overline{f}$ the algebraic isomorphism induced by $f$. If $S$ is an $\mathcal{A}$-section of $G$ then $f$ induces %the 
 a combinatorial isomorphism $f^S$ from $\mathcal{A}_S$ to $\mathcal{A}^{'}_{S^{'}}$, where $S^{'}=S^{\overline{f}}$, and $\overline{f^S}=\overline{f}^S$. Denote the set of all isomorphisms and the set of all Cayley isomorphisms from $\mathcal{A}$ to $\mathcal{A}^{'}$ that induce  given algebraic isomorphism~$\varphi$ by $\iso(\mathcal{A},\mathcal{A}^{'},\varphi)$ and $\iso_{Cay}(\mathcal{A},\mathcal{A}^{'},\varphi)$ respectively.  
 
Two permutation groups $K_1$ and $K_2$ acting on a set $\Omega$ are called \emph{2-equivalent} if $\orb(K_1,\Omega^2)=\orb(K_2,\Omega^2)$. In this case we write $K_1\approx_2 K_2$. If $\mathcal{A}=V(K,G)$ for some $K\leq \sym(G)$ containing $G_{right}$ then $\aut(\mathcal{A})$ is the largest group which is 2-equivalent to $K$. An $S$-ring $\mathcal{A}$ over $G$ is defined to be \emph{2-minimal} if
$$\{K\leq \sym(G):~K\geq G_{right}~\text{and}~K\approx_2 \aut(\mathcal{A})\}=\{\aut(\mathcal{A})\}.$$
We say that two groups $K_1,K_2\leq \aut(G)$ are \emph{Cayley equivalent} if $\orb(K_1,G)=\orb(K_2,G)$. In this case we write $K_1\approx_{Cay} K_2$. If $\mathcal{A}=\cyc(K,G)$ for some $K\leq \aut(G)$ then $\aut_G(\mathcal{A})$ is the largest group which is Cayley equivalent to $K$. A cyclotomic $S$-ring $\mathcal{A}$ over $G$ is defined to be \emph{Cayley minimal} if
$$\{K\leq \aut(G):~K\approx_{Cay} \aut_G(\mathcal{A})\}=\{\aut_G(\mathcal{A})\}.$$
It easy to see that the group ring $\mathbb{Z}G$ is 2- and Cayley minimal.  However, in general, $S$-ring can be 2-minimal but noncyclotomic. On the other hand, for example, if $G$ is elementary abelian group of order $p^n$, $L,U\leq G$, and $|L|=p,|U|=p^{n-1}$ then $\mathbb{Z}U\wr_{U/L}\mathbb{Z}(G/L)$ is Cayley minimal but not 2-minimal.

\subsection{$p$-$S$-rings}

Let $p$ be a prime number. We say that an $S$-ring $\mathcal{A}$ over a $p$-group $G$ is a \emph{$p$-$S$-ring} if every basic set of $\mathcal{A}$ has a $p$-power cardinality. In this subsection we give some properties of $p$-$S$-rings. Until the end of the subsection $G$ is a $p$-group and $\mathcal{A}$ is a $p$-$S$-ring over~$G$.

%\begin{lemm}{\em\cite[Proposition 2.13]{}}\label{p1}
%$(1)$ The thin radical $O_{\theta}(\mathcal{A})$ is nontrivial;

%$(2)$ there exists a chain of $\mathcal{A}$-subgroups
      %$$G_0=\{e\} < G_1 < \ldots < G_r=G$$
%such that $|G_{i+1}:G_i|=p$ for all $i\in \{0,\ldots, r-1\}$.
%\end{lemm}

\begin{lemm}{\em\cite[Proposition 3.4 (i)]{HM}}\label{p2}
Let $G$ be abelian. If there exists a basic set $X\in \mathcal{S}(\mathcal{A})$ with $|X|=|G|/p$  then $\mathcal{A}=\mathcal{A}_U \wr \mathcal{A}_{G/U}$, where $U\leq G$ is an $\mathcal{A}$-subgroup of index~$p$.
\end{lemm}

\begin{lemm}{\em\cite[Lemma 2.18 (i),(iii)]{FK}}\label{p3}
Let $U$ be an $\mathcal{A}$-subgroup of index~$p$ and $X\in \mathcal{S}(\mathcal{A})$. Then the following hold:

$(1)$ $X$ is contained in an $U$-coset. In particular, $\rad(X)\leq U$.

%$(2)$ Let $L$ be an $\mathcal{A}$-subgroup of order~$p$ such that $L\trianglelefteq G$ and $L\nleq \rad(X)$. Then $xL\cap X=Lx\cap X=\{x\}$ for every $x\in X$.

$(2)$ If $G$ is abelian  and $|O_{\theta}(\mathcal{A})\cap U||X|>|G|/p$ then $O_{\theta}(\mathcal{A}) \cap \rad(X)>\{e\}$.
\end{lemm}

\subsection{$CI$-$S$-rings}

Let $\mathcal{A}$ be an $S$-ring over $G$. Put 
$$\iso(\mathcal{A})=\{f\in \sym(G):~\text{f is an isomorphism from}~\mathcal{A}~\text{onto}~\text{$S$-ring over}~G\}.$$
We say that an $S$-ring $\mathcal{A}$ is a \emph{$CI$-$S$-ring} if $\iso(\mathcal{A})=\aut(\mathcal{A})\aut(G)$. This definition was suggested by Hirasaka and Muzychuk in~\cite{HM}. Also they proved, in fact, the following statement.

\begin{lemm}\label{ci1}
Let $G$ be a finite group. If every schurian $S$-ring over $G$ is a $CI$-$S$-ring then the $G$ is a  $DCI$-group.
\end{lemm}

 Further we give another equivalent definition of the $CI$-$S$-ring that is more convenient for us.

\begin{lemm}\label{ci2}
Let $\mathcal{A}$ be an $S$-ring over $G$. Then the following conditions are equivalent:

$(1)$ $\mathcal{A}$ is a $CI$-$S$-ring;

$(2)$ for every isomorphism $f$ from $\mathcal{A}$ to an $S$-ring $\mathcal{A}^{'}$ over $G$ there exists a Cayley isomorphism $\varphi$ from $\mathcal{A}$ to $\mathcal{A}^{'}$ such that $\overline{f}=\overline{\varphi}$.
\end{lemm}

\begin{proof}
Let $f$ be an isomorpism from $\mathcal{A}$ to an $S$-ring $\mathcal{A}^{'}$ over $G$. Suppose that $\mathcal{A}$ is a $CI$-$S$-ring. Then $f=f_1\varphi$, where $f_1\in \aut(\mathcal{A})$ and $\varphi\in \aut(G)$. The bijection $\varphi=f_1^{-1}f$ is a Cayley isomorphism from $\mathcal{A}$ to $\mathcal{A}^{'}$ and $\overline{\varphi}=\overline{f_1^{-1}f}=\overline{f}$ because $f_1$ induces the trivial algebraic isomorphism. Therefore, Condition~$(2)$ of the lemma holds.  
 
Conversly, suppose that Condition~$(2)$ of the lemma holds. Then there exists a Cayley isomorphism $\varphi$ from $\mathcal{A}$ to  $\mathcal{A}^{'}$ such that $\overline{f}=\overline{\varphi}$. So $f\varphi^{-1}$ is an isomorphism from $\mathcal{A}$ to itself that induces the trivial algebraic isomorphism. This means that $f\varphi^{-1}\in \aut(\mathcal{A})$ and hence $\mathcal{A}$ is a $CI$-$S$-ring. 
\end{proof}

From~\cite[Theorem 3.2]{HM} it follows that the tensor product and the $S$-wreath product with $|S|=1$ of two $CI$-$S$-rings over an elementary abelian group is a $CI$-$S$-ring.

We finish the subsection with two recent results on the $CI$-property for $S$-rings over an elementary abelian group.

\begin{lemm}{\em\cite[Proposition 3.3]{FK}}\label{quont}
Let $\mathcal{A}$ be a schurian $p$-$S$-ring over an elementary abelian group $G$ and $L\leq G$  an $\mathcal{A}$-subgroup of order $p$ such that $\mathcal{A}_{G/L}$ is 2-minimal. Then $\mathcal{A}$ is a $CI$-$S$-ring.
\end{lemm}

The next lemma is a particular case of \cite[Proposition 3.4]{FK}.

\begin{lemm}\label{cyclotom}
Let $G$ be an elementary abelian group of odd order. If for every $p$-group $K\leq \aut(G)$ with $|C_G(K)|\geq p^2$ the $S$-ring $\cyc(K,G)$ is a $CI$-$S$-ring then $G$ is a $DCI$-group
\end{lemm}

\section{Proof of Theorem~\ref{main}}

Let $\mathcal{A}^{'}$ be an $S$-ring over $G$ and $f$ an isomorphism from $\mathcal{A}$ to $\mathcal{A}^{'}$. From Lemma~\ref{simgwr} it follows that $\mathcal{A}^{'}$ is the  $U^{{f}}/L^{{f}}$-wreath product. Since $G$ is the direct product of elementary abelian groups, there exists $\theta\in \aut(G)$ such that $U^{\theta}=U^{{f}}$ and $L^{\theta}=L^{{f}}$. So replacing 
$\mathcal{A}'$ %by $\mathcal{A}^{\theta^{-1}}$ 
 with $(\mathcal{A}')^{\theta^{-1}}$  
and $f$ %by 
 with 
$f\theta^{-1}$ we may assume that $U^{{f}}=U$ and $L^{{f}}=L$. 

By the supposition, the $S$-rings  $\mathcal{A}_U$ and $\mathcal{A}_{G/L}$ are CI-$S$-rings. So by Lemma~\ref{ci2} there exist Cayley isomorphisms $\varphi_0:\mathcal{A}_U\rightarrow \mathcal{A}^{'}_U$ and $\psi_0:\mathcal{A}_{G/L}\rightarrow \mathcal{A}^{'}_{G/L}$ such that $\overline{f^U}=\overline{\varphi_0}$ and $\overline{f^{G/L}}=\overline{\psi_0}$. It is clear that 
$$\iso_{Cay}(\mathcal{A}_U,\mathcal{A}_U^{'},\overline{f^U})=\aut_U(\mathcal{A}_U)\varphi_0$$ 
and 
$$\iso_{Cay}(\mathcal{A}_{G/L},\mathcal{A}_{G/L}^{'},\overline{f^{G/L}})=\aut_{G/L}(\mathcal{A}_{G/L})\psi_0.$$
Let us show that there exist $\varphi\in \aut_U(\mathcal{A}_U)\varphi_0$ and $\psi\in\aut_{G/L}(\mathcal{A}_{G/L})\psi_0$ such that $\varphi^S=\psi^S$. Note that $\overline{\varphi_0^S}=\overline{\psi_0^S}=\overline{f^S}$. So $\varphi_0^S(\psi_0^S)^{-1}\in \aut_S(\mathcal{A}_S)$. By the condition of the theorem  $\aut_{S}(\mathcal{A}_{S})=\aut_U(\mathcal{A}_U)^{S}\aut_{G/L}(\mathcal{A}_{G/L})^{S}$. This implies that there exist $\sigma_1\in \aut_U(\mathcal{A}_U)$ and $\sigma_2\in \aut_{G/L}(\mathcal{A}_{G/L})$ such that $\varphi_0^S(\psi_0^S)^{-1}=\sigma_1^S\sigma_2^S$. Put 
$$\varphi=\sigma_1^{-1}\varphi_0~\text{and}~\psi=\sigma_2\psi_0.$$ 
The straightforward check shows that 
$$\varphi^S=(\sigma_1^S)^{-1}\varphi_0^S=(\sigma_1^S)^{-1}\sigma_1^S\sigma_2^S \psi_0^S=\psi^S.$$

Since $G$ is the direct product of elementary abelian groups, there exist groups $D$ and $V$ such that $G=D\times U$ and $U=V\times L$. Let $D=\langle x_1 \rangle \times \ldots \times \langle x_t \rangle$ and $(x_iL)^{\psi}=y_iz_iL$, where $y_i\in D,~z_i\in V$, and $i\in\{1,\ldots,t\}$. The elements $y_iz_i,~i\in\{1,\ldots,t\}$, generate %the 
 a group $D^{'}$ of rank~$t$ because $\rk(D)=t$ 
, $LD^{'}/L=(LD/L)^\psi$ and $D \cap L=D^{'}\cap L=\{e\}$. The latter equation follows from the facts that 
$D^{'} \le V \times D$ and $(V \times D) \cap L=\{e\}$. 

Note that %$D^{'}\cap U=e$. 
$D^{'}\cap U= \{e\}$  also holds.
Indeed, let $g\in D^{'}\cap U$. Then $(gL)^{\psi^{-1}}\in S\cap D/L=\{L\}$ because $S^{\psi}=S^{{f^{G/L}}}=S$. So $g\in D^{'} \cap L,$ and hence $g=e$. 
%However, $D^{'}\leq D\times V$ and hence $g\in (D\times V)\cap L=\%{e\}$. This means that $G=D^{'}\times U$. 

Since $G$ is the direct product of elementary abelian groups and $G=D\times U=D^{'}\times U$, there exists $\alpha\in \aut(G)$ such that
$$\alpha^U=\varphi,~(x_i)^{\alpha}=y_iz_i,~i\in\{1,\ldots,t\}.$$
From the definition of $\alpha$ it follows that $L^{\alpha}=L$, $U^{\alpha}=U$, and $(dL)^{\psi}=d^{\alpha}L$ for every $d\in D$. Let us check that $X^{\alpha}=X^f$ for every $X\in \mathcal{S}(\mathcal{A})$. If $X\subseteq U$ then $X^{\alpha}=X^f$ by the definition of $\alpha$. Suppose that $X$ lies outside $U$. Then $L\leq \rad(X)$ and hence $X=d_1v_1L\cup \ldots \cup d_sv_sL$, where $d_i\in D,v_i\in V,~i\in\{1,\ldots,s\}$. Clearly, 
$$X^{\alpha}=d_1^{\alpha}v_1^{\alpha}L^{\alpha}\cup \ldots \cup d_s^{\alpha}v_s^{\alpha}L^{\alpha}=d_1^{\alpha}v_1^{\varphi}L\cup \ldots~\cup d_s^{\alpha}v_s^{\varphi}L.$$ 
Since $\varphi^S=\psi^S$, we conclude that $v_i^{\varphi}L=(v_iL)^{\varphi}=(v_iL)^{\psi}$ for every $i\in\{1,\ldots,s\}$. The direct check implies the following:
\begin{eqnarray}
\nonumber X^f/L=(X/L)^{f^{G/L}}=(X/L)^{\psi}=\{d_1v_1L,\ldots,d_sv_sL\}^{\psi}=\\
\nonumber\{(d_1L)^{\psi}(v_1L)^{\psi},\ldots,(d_sL)^{\psi}(v_sL)^{\psi}\}=\{d_1^{\alpha}v_1^{\varphi}L,\ldots,d_s^{\alpha}v_s^{\varphi}L\}=X^{\alpha}/L.
\end{eqnarray}
Since $L\leq \rad(X)$ and $L\leq \rad(X^{\alpha})$, we obtain that $X^f=X^{\alpha}$. Therefore, 
$$\alpha\in \iso_{Cay}(\mathcal{A},\mathcal{A}^{'},\overline{f}).$$ 

Thus, we proved that for every isomorphism $f$ from $\mathcal{A}$ to an $S$-ring $\mathcal{A}^{'}$ over $G$ there exists a Cayley isomorphism which induces the algebraic isomorphism $\overline{f}$. So $\mathcal{A}$ is a $CI$-$S$-ring by Lemma~\ref{ci2} and the theorem is proved.

\section{Corollaries of Theorem~\ref{main}}

Throughout this section $\mathcal{A}$ is an $S$-ring over a group $G$ which is the direct product of elementary abelian groups. In the next two propositions we assume that $\mathcal{A}$ is the $S$-wreath product for some $\mathcal{A}$-section $S=U/L$ of $G$ and the $S$-rings $\mathcal{A}_U$ and $\mathcal{A}_{G/L}$ are $CI$-$S$-rings.

\begin{prop}\label{trivial}
If $\mathcal{A}_S=\mathbb{Z}S$ then $\mathcal{A}$ is a $CI$-$S$-ring.
\end{prop}

\begin{proof}
Obviously, $\aut_U(\mathcal{A}_U)^S\leq \aut_S(\mathcal{A}_S)$. On the other hand, $\aut_S(\mathcal{A}_S)$ is trivial because $\mathcal{A}_S=\mathbb{Z}S$. Thus, $\aut_U(\mathcal{A}_U)^S=\aut_S(\mathcal{A}_S)$ and we are done by Theorem~\ref{main}.
\end{proof}

\begin{prop}\label{min}
Let $\mathcal{A}$ be  cyclotomic. Suppose that $\mathcal{A}_S$ is 2-minimal or Cayley minimal. Then $\mathcal{A}$ is a $CI$-$S$-ring.
\end{prop}

\begin{proof}
Since $\mathcal{A}$ is cyclotomic, the $S$-rings $\mathcal{A}_U$ and $\mathcal{A}_S$ are also cyclotomic. Clearly, $\mathcal{A}_S=\cyc(\aut_U(\mathcal{A}_U)^S, S)$. Therefore, 
$$\aut_U(\mathcal{A}_U)^S S_{right}\approx_2 \aut(\mathcal{A}_S)~\text{and}~\aut_U(\mathcal{A}_U)^S\approx_{Cay} \aut_S(\mathcal{A}_S).$$ 
If $\mathcal{A}_S$ is Cayley minimal then $\aut_U(\mathcal{A}_U)^S=\aut_S(\mathcal{A}_S)$ and we are done by Theorem~\ref{main}. Suppose that $\mathcal{A}_S$ is 2-minimal. Then  
$$\aut_U(\mathcal{A}_U)^S S_{right}=\aut(\mathcal{A}_S)\geq \aut_S(\mathcal{A}_S) S_{right}.$$
Each of the subgroups $\aut_U(\mathcal{A}_U)^S$ and $\aut_S(\mathcal{A}_S)$ intersects trivially with $S_{right}$. %So 
 This shows that
$|\aut_U(\mathcal{A}_U)^S|\geq |\aut_S(\mathcal{A}_S)|$. On the other hand, obviously, $\aut_U(\mathcal{A}_U)^S\leq \aut_S(\mathcal{A}_S)$. Thus, $\aut_U(\mathcal{A}_U)^S=\aut_S(\mathcal{A}_S)$ and Theorem~\ref{main} implies that $\mathcal{A}$ is a $CI$-$S$-ring.
\end{proof}

If $(g_1,\ldots,g_n)\in G^n$ and $\varphi\in \aut(G)$ then put $(g_1,\ldots,g_n)^{\varphi}=(g_1^{\varphi},\ldots,g_n^{\varphi})$. In the following two propositions we assume that $G$ is the elementary abelian group of order~$p^n$, where $p$ is a prime and $n\geq 1$, and $\mathcal{A}$ is a $p$-$S$-ring over $G$. 

\begin{prop}\label{thin}
Suppose that $|G:O_{\theta}(\mathcal{A})|=p$. Then 

$(1)$ $\mathcal{A}=\mathbb{Z}O_{\theta}(\mathcal{A}) \wr_{O_{\theta}(\mathcal{A})/L} \mathbb{Z}(G/L)$ for some $\mathcal{A}$-subgroup $L\leq O_{\theta}(\mathcal{A})$;

$(2)$ $\mathcal{A}$ is a cyclotomic Cayley minimal $CI$-$S$-ring.
\end{prop}

\begin{proof}
Let $X$ be a basic set of $\mathcal{A}$ outside $O_{\theta}(\mathcal{A})$ and $|X|=p^k$. Put $L=\rad(X)$. Assume that $|L|=p^l<p^k$. Statement~$(1)$ of Lemma~\ref{p3} yields that $L\leq O_{\theta}(\mathcal{A})$. Let $\pi:G\rightarrow G/L$ be the canonical epimorphism. The set $\pi(X)$ is a basic set with the trivial radical of the $S$-ring $\mathcal{A}_{G/L}$  and $|\pi(X)|\geq p$. Note that $|G/L|=p^{n-l}$ and $|O_{\theta}(\mathcal{A}_{G/L})|=p^{n-l-1}$. So $|\pi(X)||O_{\theta}(\mathcal{A}_{G/L})|>|G/L|/p$. We obtain a contradiction with Statement~$(2)$ of Lemma~\ref{p3} for $U=O_{\theta}(\mathcal{A}_{G/L})$. Therefore, $X$ is an $L$-coset. From  Lemma~\ref{burn} it follows that every basic set of $\mathcal{A}$ outside  $O_{\theta}(\mathcal{A})$ is of the form $gX^{(m)}$, where $g\in O_{\theta}(\mathcal{A})$ and $m$ is an integer coprime to $p$. This implies that  every basic set  of $\mathcal{A}$ outside $O_{\theta}(\mathcal{A})$ is an $L$-coset and hence $\mathcal{A}$ is the $O_{\theta}(\mathcal{A})/L$-wreath product. It is clear that $\mathcal{A}_{O_{\theta}(\mathcal{A})}=\mathbb{Z}O_{\theta}(\mathcal{A})$ and $\mathcal{A}_{G/L}=\mathbb{Z}(G/L)$. Thus, Statement~$(1)$ of the proposition is proved.

Obviously, $\mathcal{A}_{O_{\theta}(\mathcal{A})}$ and $\mathcal{A}_{G/L}$ are $CI$-$S$-rings. It follows also that  $\mathcal{A}_{O_{\theta}(\mathcal{A})/L}=\mathbb{Z}(O_{\theta}(\mathcal{A})/L)$ and hence $\mathcal{A}$ is a $CI$-$S$-ring by Proposition~\ref{trivial}. Let us prove that $\mathcal{A}$ is  Cayley minimal. Let $g_1,\ldots,g_{n-1}$ be generators of $O_{\theta}(\mathcal{A})$ and $x\in X$. Then 
$$\aut_G(\mathcal{A})=\{\sigma\in \aut(G):~(g_1,\ldots,g_{n-1},x)^{\sigma}=(g_1,\ldots,g_{n-1},xl),l\in L\}$$ 
and $\mathcal{A}=\cyc(\aut_G(\mathcal{A}),G)$.  Besides, $|\aut_G(\mathcal{A})|=|L|$. If $K\approx_{Cay} \aut_G(\mathcal{A})$ then $|K|\geq |L|$ because $X$  is an orbit of $K$. Therefore, $K=\aut_G(\mathcal{A})$. This means that $\mathcal{A}$ is  Cayley minimal.
\end{proof}

\begin{prop}\label{easy}
Let $\mathcal{A}$ be a cyclotomic $p$-$S$-ring which is the $S$-wreath product for some $\mathcal{A}$-section $S=U/L$ of $G$. Suppose that the $S$-rings $\mathcal{A}_U$ and $\mathcal{A}_{G/L}$ are $CI$-$S$-rings and $|G:O_{\theta}(\mathcal{A})|=p^2$. Then $\mathcal{A}$ is a $CI$-$S$-ring.
\end{prop}

\begin{proof}
Since every basic set outside $U$ has %the 
nontrivial radical, we obtain that $O_{\theta}(\mathcal{A})\leq U$ and hence $|S:O_{\theta}(\mathcal{A}_S)|\in \{1,p\}$. If $|S:O_{\theta}(\mathcal{A}_S)|=1$ then $\mathcal{A}_S=\mathbb{Z}S$ and $\mathcal{A}$ is a $CI$-$S$-ring by Proposition~\ref{trivial}. If $|S:O_{\theta}(\mathcal{A}_S)|=p$ then $\mathcal{A}_S$ is Cayley minimal by Statement~$(2)$ of Proposition~\ref{thin}. Therefore, Proposition~\ref{min} implies that $\mathcal{A}$ is a $CI$-$S$-ring.
\end{proof}

\section{Decomposable $S$-rings over elementary abelian groups of small ranks}

Let $p$ be an odd prime and $G$   be 
an elementary abelian group of order $p^n,~n\geq 1$. These notations are valid until the end of the paper. In view of Lemma~\ref{cyclotom} to  prove that $G$ is a $DCI$-group it is sufficient to show every cyclotomic $p$-$S$-ring over $G$ is a $CI$-$S$-ring. In fact, it was proved that every cyclotomic $p$-$S$-ring over $G$ is a $CI$-$S$-ring for $n=4$ in~\cite{HM} and for $n=5$ in~\cite[Theorem~5.2]{FK}. One of the main difficulties in the proofs is to check the $CI$-property for decomposable $S$-rings. For example, in the case $n=5$ the proof of the fact that every decomposable cyclotomic $p$-$S$-ring is a $CI$-$S$-ring takes 10 pages. The main goal of this section is to give a  short proof of this fact for $n\leq 5$ using Theorem~\ref{main}. 

We start the section with the description of all $p$-$S$-rings over an elementary abelian group of rank at most~$3$. All $p$-$S$-rings over an elementary abelian group of rank at most~$2$ and all schurian $p$-$S$-rings over the elementary abelian group of rank~$3$ were described in~\cite[p.14-15]{HM}. Later in~\cite{Sp2}, it was proved that every $p$-$S$-ring over the elementary abelian group of rank~$3$ is schurian. The next lemma summarizes all these results.

\begin{lemm}\label{rank3}
Let $n\leq 3$ and $\mathcal{A}$  be a $p$-$S$-ring over $G$. Then $\mathcal{A}$ is cyclotomic and 

$(1)$ if $n=1$ then $\mathcal{A}\cong_{Cay}\mathbb{Z}C_p$;

$(2)$ if $n=2$ then $\mathcal{A}\cong_{Cay}\mathbb{Z}C_p$ or $\mathcal{A}\cong_{Cay}\mathbb{Z}C_p\wr \mathbb{Z}C_p$;

$(3)$ if $n=3$ then $\mathcal{A}$ is  one of the $S$-rings given in Table~$1$ up to Cayley isomorphism.
\end{lemm}

\begin{table}

{\small
\begin{tabular}{|l|l|l|l|}
  \hline
  no. & $\mathcal{A}$ & decomposable & $|O_{\theta}(\mathcal{A})|$    \\
  \hline
  $1.$ & $\mathbb{Z}C_p^3$ & no & $p^3$\\ \hline
  $2.$ & $\mathbb{Z}C_p^2\wr \mathbb{Z}C_p$  & yes & $p^2$ \\  \hline
  $3.$ & $\mathbb{Z}C_p\wr \mathbb{Z}C_p^2$  & yes & $p$ \\  \hline
	$4.$ & $(\mathbb{Z}C_p\wr \mathbb{Z}C_p)\otimes \mathbb{Z}C_p$  & yes & $p^2$\\  \hline
	$5.$ & $\mathbb{Z}C_p\wr \mathbb{Z}C_p\wr \mathbb{Z}C_p$  & yes  & $p$\\  \hline
	$6.$ & $\cyc(\langle \sigma \rangle,C_p^3),~\sigma=\left(\begin{smallmatrix} 1& 1& 0\\ 0& 1& 1\\ 0& 0& 1\end{smallmatrix}\right)$ & no  & $p$\\ \hline

\end{tabular}
}
\caption{$p$-$S$-rings over $C_p^3$ for an odd prime $p$.}
\end{table}
%Table 1.

\begin{lemm}{\em\cite[Lemma 2.17]{FK}}\label{except}
Let $n=3$ and $\mathcal{A}$  be a $p$-$S$-ring over $G$ which is Cayley isomorphic to the $S$-ring no.~$6$ from Table~$1$. Then $|\aut(\mathcal{A})|=p^4$ and $|\aut(\mathcal{A})_e|=p$.
\end{lemm}

\begin{lemm}\label{2min}
Let $n\leq 3$ and $\mathcal{A}$   be an indecomposable  $p$-$S$-ring over $G$. Then $\mathcal{A}$ is 2-minimal.
\end{lemm}
\begin{proof}
If $\mathcal{A}=\mathbb{Z}G$ then, obviously, $\mathcal{A}$ is 2-minimal. If $\mathcal{A}\neq \mathbb{Z}G$ then $n=3$ and $\mathcal{A}$ is Cayley isomorphic to the $S$-ring no.~$6$ from Table~$1$. In this case the statement of the lemma follows from Lemma~\ref{except}.
\end{proof}

\begin{lemm}{\em\cite[Theorem 4.1]{FK}}\label{rank4}
Let $n=4$ and $\mathcal{A}$  be an indecomposable schurian $p$-$S$-ring over $G$. Then $\mathcal{A}$ is 2-minimal.
\end{lemm}

\begin{lemm}\label{caymin}
Let  $n\leq 3$ and $\mathcal{A}$  be a $p$-$S$-ring over $G$. Then $\mathcal{A}$ is Cayley minimal except for the case when $n=3$ and $\mathcal{A}\cong_{Cay} \mathbb{Z}C_p\wr \mathbb{Z}C_p\wr \mathbb{Z}C_p$ (the $S$-ring no. 5 from Table~$1$).
\end{lemm}

\begin{proof}
The statement of the lemma is obvious when $\mathcal{A}=\mathbb{Z}G$. If $|O_{\theta}(\mathcal{A})|=p^{n-1}$ then the statement of the lemma follows from Statement~$(2)$ of Proposition~\ref{thin}. If $|O_{\theta}(\mathcal{A})|<p^{n-1}$ then $n=3$ and $\mathcal{A}$ is Cayley isomorphic to one of the $S$-rings no.~3, 5, 6 from Table~$1$. If $\mathcal{A}$ is Cayley isomorphic to the $S$-ring no.~$6$ from Table~$1$ then Lemma~\ref{except} yields that $|\aut_G(\mathcal{A})|=p$. Since  $\mathcal{A}$ is cyclotomic and nontrivial, we conclude that $\mathcal{A}$ is Cayley minimal. 

Suppose that  $\mathcal{A}\cong_{Cay}\mathbb{Z}C_p\wr \mathbb{Z}C_p^2$ (the $S$-ring no. 3 from Table~$1$). Put $O_{\theta}(\mathcal{A})=\langle a \rangle=A$. Every basic set of $\mathcal{A}$ outside $A$ is an $A$-coset. Let $b,c\in G\setminus A$ such that $cA$ is a basic set outside $\langle b,A \rangle$. If $\varphi\in \aut_G(\mathcal{A})$ fixes $b$ and $c$ then $\varphi$ is trivial. So $|\aut_G(\mathcal{A})|\leq p^2$. On the other hand, the direct check implies that $\aut_G(\mathcal{A})$ contains the following subgroup
$$\{\varphi\in \aut(G):~(a,b,c)^{\varphi}=(a,ba^k,ca^l),~k,l=0,\ldots,p-1\}$$
of order $p^2$. Therefore, $|\aut_G(\mathcal{A})|=p^2$. Assume that $\mathcal{A}$ is not Cayley minimal. Then there exists a group $M\leq \aut_G(\mathcal{A})$ of order $p$ such that $\mathcal{A}=\cyc(M,G)$. Let $\psi$ be a generator of $M$ and $(a,b,c)^{\psi}=(a,ba^k,ca^l)$, where $(k,l)\neq (0,0)$. There exist $i,j\in\{0,\ldots,p-1\}$ such that $(i,j)\neq (0,0)$ and $ki+lj=0$. Then $(b^ic^j)^{\psi}=b^ic^ja^{ki+lj}=b^ic^j$ and hence $(b^ic^j)^{M}=\{b^ic^j\}$. Since $b^ic^jA\in \mathcal{S}(\mathcal{A})$, we obtain a contradiction with $\mathcal{A}=\cyc(M,G)$. Thus, $\mathcal{A}$ is Cayley minimal and the lemma is proved.
\end{proof}

\begin{lemm}\label{notcayleymin}
Let $\mathcal{A}\cong_{Cay}\mathbb{Z}C_p\wr \mathbb{Z}C_p\wr \mathbb{Z}C_p$. Then $|\aut_G(\mathcal{A})|=p^3$ and $\mathcal{A}$ is not Cayley minimal.
\end{lemm}

\begin{proof}
Let $O_{\theta}(\mathcal{A})=\langle a \rangle=A$ and  $b,c\in G\setminus A$ such that $bA\in \mathcal{S}(\mathcal{A})$ and  $c(A\times B)$ is a basic set of $\mathcal{A}$ outside $\langle b,A \rangle$. If $\varphi\in \aut_G(\mathcal{A})$ fixes $b$ and $c$ then $\varphi$ is trivial. So $|\aut_G(\mathcal{A})|\leq p^3$. The direct check yields that $\aut_G(\mathcal{A})$ contains the following subgroup
$$\{\varphi\in \aut(G):~(a,b,c)^{\varphi}=(a,ba^i,cb^ja^k),~i,j,k=0,\ldots,p-1\}$$ 
of order $p^3$. Therefore, $|\aut_G(\mathcal{A})|=p^3$. Define $\psi_1,\psi_2\in \aut(G)$ as follows:
$$(a,b,c)^{\psi_1}=(a,ba,cb),~(a,b,c)^{\psi_2}=(a,b,ca).$$
Put $M=\langle \psi_1,\psi_2\rangle$. It can be checked in the straightforward way  that $|M|=p^2$ and $\mathcal{A}=\cyc(M,G)$. Thus, $\mathcal{A}$ is not Cayley minimal. 
\end{proof}

\begin{prop}\label{reprove}
Let $n\leq 5$ and $\mathcal{A}$   be 
a cyclotomic $p$-$S$-ring  over $G$ such that $\mathcal{A}$ is the $S$-wreath product for some $\mathcal{A}$-section $S=U/L$ of $G$. Suppose that every elementary abelian group of  rank at most $n-1$ is a $DCI$-group. Then $\mathcal{A}$ is a $CI$-$S$-ring.
\end{prop}

\begin{proof}
The $S$-rings $\mathcal{A}_U$ and $\mathcal{A}_{G/L}$ are $CI$-$S$-rings because every elementary abelian group of  rank at most $n-1$ is a $DCI$-group. Clearly, $\mathcal{A}_S$ is a cyclotomic $p$-$S$-ring. Since $n\leq 5$, we have  $|S|\leq p^3$. If $\mathcal{A}_S\ncong_{Cay}\mathbb{Z}C_p\wr \mathbb{Z}C_p\wr \mathbb{Z}C_p$ then $\mathcal{A}_S$ is Cayley minimal by Lemma~\ref{caymin} and hence $\mathcal{A}$ is a $CI$-$S$-ring by Proposition~\ref{min}. So we may assume that $n=5$ and $\mathcal{A}_S\cong_{Cay}\mathbb{Z}C_p\wr \mathbb{Z}C_p\wr \mathbb{Z}C_p$. In this case $|L|=p$ and $|U|=p^4$. If $\mathcal{A}_{G/L}$ is indecomposable then Lemma~\ref{rank4} implies that $\mathcal{A}_{G/L}$ is 2-minimal. So $\mathcal{A}$ is a $CI$-$S$-ring by Lemma~\ref{quont}. Thus, we may assume that $\mathcal{A}_{G/L}$ is decomposable, namely $\mathcal{A}_{G/L}$ is the $U_1/L_1$-wreath product, where $L_1$ is nontrivial and $U_1<G/L$.

Put  $O_{\theta}(\mathcal{A}_S)=\langle a \rangle=A$. Let $b\in S\setminus A$ such that $bA\in \mathcal{S}(\mathcal{A}_S)$, $B=\langle b \rangle$, and  $c\in S\setminus (A\times B)$ such that $c(A\times B)\in \mathcal{S}(\mathcal{A}_S)$. If every basic set of $\mathcal{A}_{G/L}$ outside $S$ has the nontrivial radical then $A$ is the smallest nontrivial $\mathcal{A}_{G/L}$-subgroup. So $A\leq \rad(X)$ for every $X\in \mathcal{S}(\mathcal{A}_{G/L})$ outside $A$. This implies that the group $\pi^{-1}(A)$, where $\pi:G\rightarrow G/L$ is the canonical epimorphism, lies in the radical of every basic set of $\mathcal{A}$ outside $U$. Therefore, $\mathcal{A}$ is the $U/\pi^{-1}(A)$-wreath product. Note that $|\pi^{-1}(A)|=p^2$ and hence $|U/\pi^{-1}(A)|=p^2$. The $S$-rings $\mathcal{A}_U$ and $\mathcal{A}_{G/\pi^{-1}(A)}$ are $CI$-$S$-rings  by assumption of the proposition. Lemma~\ref{caymin} yields that $\mathcal{A}_{U/\pi^{-1}(A)}$ is Cayley minimal and we are done by Proposition~\ref{min}. Thus, we may assume that there exists $X\in \mathcal{S}(\mathcal{A}_{G/L})$ outside $S$ with $|\rad(X)|=1$. 

Due to $|G/L|=p^4$, we have $|X|\in\{1,p,p^2,p^3\}$. Note that 
$$\langle X \rangle \leq U_1<G/L~\eqno(2)$$ 
because every basic set outside $U_1$ has %the 
nontrivial radical. Also 
$$|\langle X \rangle|>p|X|~\eqno(3)$$
whenever $|X|\geq p$ since otherwise $|\rad(X)|=|X|\geq p$ by Lemma~\ref{p2}. From Eqs.~$(2)$ and $(3)$ it follows that $|X|\in \{1,p\}$. Let $|X|=1$. In this case $X=\{x\}$ for some $x\in (G/L)\setminus S$. Lemma~\ref{tenspr} implies that $\mathcal{A}_{G/L}=\mathcal{A}_S\otimes \mathcal{A}_{\langle x \rangle}$. Let $\varphi\in \aut_S(\mathcal{A}_S)$. Define $\psi\in \aut(G/L)$ in the following way: $\psi^S=\varphi,~x^{\psi}=x$. Then $\psi\in \aut_{G/L}(\mathcal{A}_{G/L})$. We obtained that $\aut_{G/L}(\mathcal{A}_{G/L})^S\geq \aut_S(\mathcal{A}_S)$, and therefore, $\aut_{G/L}(\mathcal{A}_{G/L})^S=\aut_S(\mathcal{A}_S)$. So $\mathcal{A}$ is a $CI$-$S$-ring by Theorem~$1$. Furthermore, we may assume that there are no basic sets of size~$1$ outside $S$ and hence $O_{\theta}(\mathcal{A}_{G/L})=A$. In this case $A$ is the smallest nontrivial $\mathcal{A}_{G/L}$-subgroup.

Now let $|X|=p$. From Eqs.~$(2)$ and $(3)$ it follows that $|\langle X \rangle|=p^3$. The group $\langle X \rangle \cap S$ is an $\mathcal{A}_{G/L}$-subgroup of order $p^2$. On the other hand, $A\times B$ is the unique $\mathcal{A}_{G/L}$-subgroup of order $p^2$ in $S$. So $\langle X \rangle \cap S=A\times B$. In view of Lemma~\ref{rank3}, the $S$-ring $\mathcal{A}_{\langle X \rangle}$ is Cayley isomorphic to the $S$-ring no.~$6$ from Table~$1$. Therefore, we may assume that
$$X=x\{b^ia^{\frac{i(i-1)}{2}}\},~i=0,\ldots,p-1,$$
for some $x\in (G/L)\setminus S$. Let $Y\in \mathcal{S}(\mathcal{A}_{G/L})$ outside $\langle X \rangle \cup S$. Assume that $\rad(Y)$ is trivial. Then $X,Y\subseteq U_1$. However, $|\langle X,Y \rangle|>|\langle X \rangle|=p^3\geq |U_1|$, a contradiction. This yields that $A\leq \rad(Y)$. Let $\pi_1:G/L\rightarrow (G/L)/A$ be the canonical epimorphism. Consider the $S$-ring $\mathcal{A}_{(G/L)/A}$ over the group $(G/L)/A$ of order $p^3$. Note that $|\pi_1(X)|=|\pi_1(c(A\times B))|=p$ and $\rad(\pi_1(X))=\rad(\pi_1(c(A\times B)))=\pi_1(B)$. The description of all $p$-$S$-rings over $C_p^3$ given in Table~$1$ implies that $\mathcal{A}_{(G/L)/A}$ is Cayley isomorphic to $\mathbb{Z}C_p\wr \mathbb{Z}C_p^2$. So $|\pi_1(Y)|=p$ and $\rad(\pi_1(Y))=\pi_1(B)$. Since $A\leq \rad(Y)$, we conclude that $Y$ is an 
%$A\times B$-coset
 $(A\times B)$-coset. 
 Thus, we proved that every basic set %of 
 in $\mathcal{S}(\mathcal{A}_{G/L})$ outside 
%$\langle X \rangle \cup U/L$ 
 $\langle X \rangle=U_1$
is an %$A\times B$-coset
$(A\times B)$-coset. The direct check shows that $\aut_{G/L}(\mathcal{A}_{G/L})$ contains the following subgroup:
$$M=\{\varphi\in \aut(G/L):~(a,b,c,x)^{\varphi}=(a,ba^i,ca^jb^k,xb^ia^{\frac{i(i-1)}{2}}),i,j,k=0,\ldots,p-1\}.$$
Therefore, $|\aut_{G/L}(\mathcal{A}_{G/L})|\geq |M|=p^3$. 

Suppose that $\varphi\in \aut_{G/L}(\mathcal{A}_{G/L})$ acts trivially on $S$. If $\varphi$ is nontrivial then from Statement~$(1)$ of Proposition~\ref{thin} it follows that $\cyc(\langle \varphi \rangle, G/L)$ is the generalized wreath product of two group rings. But this is impossible because  $\rad(X)$ is trivial. So $\varphi$ is trivial  and hence $|\aut_{G/L}(\mathcal{A}_{G/L})^S|=|\aut_{G/L}(\mathcal{A}_{G/L})|$. Using this and  Lemma~\ref{notcayleymin}, we conclude that $|\aut_{G/L}(\mathcal{A}_{G/L})^S|\geq p^3=|\aut_S(\mathcal{A}_S)|$. Thus, $\aut_{G/L}(\mathcal{A}_{G/L})^S=\aut_S(\mathcal{A}_S)$ and $\mathcal{A}$ is a $CI$-$S$-ring by Theorem~\ref{main}.
\end{proof}

Now from~\cite[Theorem~1.3]{CLi} and Proposition~\ref{reprove} it follows that every decomposable cyclotomic  $q$-$S$-ring over an elementary abelian group of rank at most~5 is a $CI$-$S$-ring  for any prime number $q$.

%\begin{rem}
%In fact, above we reprove \cite[Proposition 3.15]{HM} and \cite[Theorem 5.2]{FK}.
%\end{rem}

\begin{prop}\label{reprove1}
Let $n=6$ and $\mathcal{A}$   be a cyclotomic $p$-$S$-ring over $G$ such that $\mathcal{A}$ is the $S$-wreath product for some $\mathcal{A}$-section $S=U/L$ 
of $G$. Suppose that $\mathcal{A}_S$ is indecomposable. Then $\mathcal{A}$ is a $CI$-$S$-ring.
\end{prop}

\begin{proof}
Since the group $C_p^k$ is a $DCI$-group for $k\leq 5$,  the $S$-rings $\mathcal{A}_U$ and $\mathcal{A}_{G/L}$ are $CI$-$S$-rings. Note that $|S|\leq p^4$ because $n=6$. From Lemma~\ref{2min} and Lemma~\ref{rank4} it follows that $\mathcal{A}_S$ is 2-minimal. Therefore, $\mathcal{A}$ is a $CI$-$S$-ring by Proposition~\ref{min}.
\end{proof}
 
%We finish the paper with the following question.

%\begin{ques}
%Is it true that every decomposable cyclotomic $S$-ring over $G=C_p^6$, where $p$ is an odd prime, is a $CI$-$S$-ring?
%\end{ques}

\end{document}